\documentclass[12pt]{amsart}

\usepackage{amsmath,amsfonts,amsthm,amssymb,multicol,enumerate,parskip}
\usepackage[all]{xy}

\usepackage{hyperref}
\hypersetup{
	colorlinks,
	citecolor=black,
	filecolor=black,
	linkcolor=black,
	urlcolor=black
}

\newcommand{\mb}{\mathbb}
\newcommand{\mc}{\mathcal}

\newcommand{\eps}{\epsilon}

\newcommand\blfootnote[1]{%
  \begingroup
  \renewcommand\thefootnote{}\footnote{#1}%
  \addtocounter{footnote}{-1}%
  \endgroup
}

\renewcommand{\d}{\partial}

\newcommand{\on}{\operatorname}

\newtheorem{thm}{Theorem}

\newtheorem{lemma}{Lemma}
\newtheorem{cor}{Corollary}

\title{Bounding eigenvalues with packing density}
\author{Neal Coleman}

\begin{document}

\maketitle

\begin{abstract}
We prove a lower bound on the eigenvalues $\lambda_k$, $k\in\mb{N}$, of the Dirichlet Laplacian of a bounded domain $\Omega\subset\mb{R}^n$ of volume $V$:
$$ \lambda_k \geq C_n\bigg( \delta\frac{k}{V}\bigg)^{2/n} $$
where $\delta$ is a constant that measures how efficiently $\Omega$ can be packed into $\mb{R}^n$ and $C_n$ is the constant found in Weyl's law. This generalizes a result of Urakawa in 1984. If $\delta^{2/n} > n/(n+2)$, this bound is stronger than the eigenvalue bound proven by Li and Yau in 1983. For example, in the case of convex planar domains, we have for all $k\in\mb{N}$,
$$ \lambda_k \geq \frac{2\sqrt{3}\pi k}{V}. $$
\end{abstract}


\blfootnote{The author thanks Iosif Polterovich for bringing the paper of Urakawa to his attention and for further communication. The author thanks Chris Judge, Dylan Thurston, and Sugata Mondal for useful discussion.}

\section{Introduction}

The purpose of this paper is to prove the following theorem:
\begin{thm}\label{main-thm}
Let $\Omega$ be a bounded domain in $\mathbb{R}^n$ with Lipschitz boundary. Let $V$ be the volume of $\Omega$ and $\delta$ be its packing constant. Let $\lambda_k$ be its $k^{th}$ Dirichlet eigenvalue. Then for all $k\in\mb{N}$,
$$ \lambda_k \geq C_n\bigg(\delta\frac{k}{V}\bigg)^{2/n}.$$
\end{thm}

Here is a brief survey of the relevant literature.

Weyl \cite{Weyl1911} proved the asymptotic estimate
\begin{equation}\label{weyl-law}
\lim_{k\to\infty}\frac{\lambda_k}{k^{2/n}} = C_n\bigg(\frac{1}{V}\bigg)^{2/n}  
\end{equation}
as $k\to\infty$. Here $C_n = (2\pi)^2/(\mbox{volume of unit $n$-ball})^{2/n}.$

Polya \cite{Polya1961} considered domains which tile Euclidean space by reflection and translation. He used this tiling property and Weyl's law to prove that, for such domains,
\begin{equation}\label{polya}
\lambda_k \geq C_n\bigg(\frac{k}{V}\bigg)^{2/n} 
\end{equation}
for every $k\in\mb{N}$. He conjectured that this holds for any planar domain. This conjecture is still open.

Urakawa \cite{Urakawa1984} proved a weaker version of Theorem 1, bounding Dirichlet eigenvalues in terms of the \emph{lattice packing constant} of a domain.

Li and Yau \cite{Li1983} proved for an arbitrary domain in $\mb{R}^n$ that
\begin{equation}\label{li-yau} \lambda_k \geq \frac{n}{n+2}C_n \bigg(\frac{k}{V}\bigg)^{2/n} 
\end{equation}
for every $k\in\mb{N}$. They proved this as a corollary of an inequality about the Riesz mean of the eigenvalue sequence, $\frac{1}{k}\sum_{j=1}^k \lambda_j$. Kroger \cite{Kroger1994} proved the corresponding Li-Yau inequality for Neumann eigenvalues. 

Recent improvements to inequality \eqref{li-yau} have proceeded by adding terms to the inequality involving Riesz means of eigenvalues and extending the inequality to more general settings. For details, we refer the reader to \cite{Laptev1997}, \cite{Melas2003}, \cite{Wei2010}, \cite{Geisinger2011}, \cite{Hatzinikitas2013}, \cite{Yolcu2013a}, \cite{Yolcu2013}, \cite{Yolcu2014}, and \cite{Kovarik2015}.

Following Polya's original argument and Urakawa's generalization, we use packings of $\Omega$. By doing so, the packing constant $\delta$ of $\Omega$ enters the inequality. In fact, we are able to replace $n/(    n+2)$ in Li-Yau's estimate with $\delta^{2/n}$. This replacement sacrifices universality, but strengthens the inequality for domains with high packing constant.

\section{Discussion}

Theorem 1 is a generalization of Polya's theorem, as the packing constant of a tiling domain is $1$. Theorem 1 also permits us to replace the factor $n/(n+2)$ in inequality \eqref{li-yau} with the factor $\delta^{2/n}$. In particular, if $\delta > [n/(n+2)]^{n/2}$, then the inequality in Theorem 1 is stronger than inequality \eqref{li-yau}. Such domains are not difficult to construct; for instance, Theorem 1 is stronger than inequality \eqref{li-yau} for any domain in dimension $n\geq 2$ which has a bounding parallelopiped with less than twice the volume of the domain.

General lower bounds for various classes of domain in all dimensions tend to be weak. For instance, a theorem of Minkowski-Hlawka guarantees that the packing constant for a convex, centrally symmetric domain in $\mb{R}^n$ is no less than $\zeta(n)/2^{n-1}$. Schmidt then proved that there is a constant $c$ such that a convex domain in $\mb{R}^n$ has $\delta \geq cn^{3/2}/4^n$. See, for instance, the discussion in \cite{Toth1993}.

However, dimensions $2$ and $3$ are better-studied. As an example, we reproduce a portion of a table from the survey \cite{Bezdek2013}, modifying the last row with information from section 8.4 of the same paper:

\begin{tabular}{|c|c|}
\hline
Body & Lower bound for packing density \\ 
\hline
Unit ball & $\frac{\pi}{\sqrt{18}} = 0.7408\ldots$ \\
Regular octahedron & $\frac{18}{19} = 0.9473\ldots$ \\
Cylinder over a plane domain $K$ & $\delta(K)$ \\
Doubled cone & $\pi\sqrt{6}/9 = 0.855\ldots$ \\
Tetrahedron & $ 0.856\ldots $ \\
\hline
\end{tabular}

Here a cylinder over a plane domain $K$ is the Minkowski sum of $K\times\{0\}$ with a line segment $s$ (which is assumed non-parallel to $K$). Observe that the packing constant of cylinders implies the three-dimensional case of Laptev's proof of Polya's conjecture in \cite{Laptev1997} for products of tiling domains and arbitrary domains.

In \cite{Torquato2009}, the authors provide a survey of known lattice packing constants for Platonic and Archimedean solids. All the Platonic and Archimedean solids have packing densities in excess of $0.5$.

Specializing to $n=2$, Kuperberg-Kuperberg \cite{Kuperberg1990}, later improved by Doheny \cite{Doheny1995}, found lower bounds for packing constants of convex planar domains: 
\begin{thm}
If $\Omega$ is a convex planar domain, then its packing constant is at least $\sqrt{3}/2$.
\end{thm}
This gives the following corollary to Theorem 1.
\begin{cor}
Let $\Omega$ be a convex planar domain. Then for all $k$, its Dirichlet eigenvalues satisfy
$$ \lambda_k > 2\sqrt{3}\pi\frac{k}{V}. $$
\end{cor}
Note that this improves inequality \eqref{li-yau} by a factor of $\sqrt{3}$. (After the author uploaded a first draft of this preprint to the arxiv, Iosif Polterovich informed the author that this result is known to him and Olivier Mercier.)

We also note that in general, packing constants are greater than lattice packing constants. For instance, the regular tetrahedron has a low lattice packing constant and admits non-lattice packings with much, much higher density; see section 8 of \cite{Bezdek2013} for more information. In fact, many domains have high packing constants but relatively low lattice packing constants. This is the case even in $\mathbb{R}^2$; there are families of convex polygons (such as triangles) which tile the plane, but whose lattice packing constants are strictly less than one.

\section{Proof of Result}

Let $\Omega$ be an open, bounded domain in $\mathbb{R}^n$ with volume $V = |\Omega|$. Let $\Delta = -\sum\d_i^2$ denote the Laplace operator in $\Omega$. Denote by
$$0<\lambda_1<\lambda_2\leq\cdots$$ 
the spectrum of the Laplace operator with Dirichlet boundary conditions.

The proof proceeds in two steps. The first step applies Dirichlet domain monotonicity and uses Weyl's law to prove Theorem 1. The second step equates the limit $\lim_{\eps\to 0}N(\eps)\eps^n$ with the packing constant of $\Omega$.

Now we prove Theorem 1.

\begin{proof}Let $\sigma > 0$ be given. Set $G = [-\sigma/2,\sigma/2]^n$. Note that $|G| = \sigma^n$. Denote by $N_G$ and $N_\Omega$ the eigenvalue counting functions of $G$ and $\Omega$, resp.

Let $\mc{P}$ be a maximal packing of $\Omega$. (This exists by work of Groemer, \cite{Groemer1986}.) Let $N(\sigma)$ be the number of components of $\mc{P}$ contained within $\Omega$. By Dirichlet domain monotonicity, for every $x$,
$$ \sum_{A\in\mc{P}, A\subset G} N_A(x) \leq N_G(x). $$

Since $G$ tiles $\mb{R}^n$, by Polya's theorem, inequality \eqref{polya}, we have 
$$ N_\Omega(x) \leq \frac{N_G(x)}{N(\sigma)} \leq C_n \frac{\sigma^n}{N(\sigma)} x^{n/2}. $$
This is true for every $\sigma > 0$. Letting $\sigma\to\infty$ and using the lemma proved below, 
$$ N_\Omega(x) \leq C_n \frac{V}{\delta} x^{n/2} $$
for all $x$. Equivalently,
$$ \lambda_k(\Omega) \geq C_n \bigg( \delta \frac{k}{V} \bigg)^{2/n} $$
for all $k$ (where $C_n$ is a different constant depending only on dimension).

This completes the proof. 
\end{proof}

\begin{lemma}
Let $\Omega$ be a bounded domain in $\mathbb{R}^n$. Let $\mc{P}$ be a maximum-density packing of $\Omega$. For $\sigma > 0$, let $N(\sigma)$ be the number of components of $\mc{P}$ entirely contained in $[-\sigma/2,\sigma/2]^n$. Then
$$ \lim_{\sigma\to\infty} \frac{\sigma^n}{N(\sigma)} = \frac{V}{\delta}. $$
\end{lemma}

\begin{proof}
For a more thorough summary of basic concepts in packing, we refer the reader to \cite{Toth1993}, section 2, and to \cite{Groemer1986}. 

A packing $\mc{P}$ of $\mb{R}^n$ by $\Omega$ is a collection pairwise disjoint congruent copies of $\Omega$. If $G$ is a domain, define the inner density $d_{\on{inn}}$ and outer density $d_{\on{out}}$ with respect to $G$ to be
$$ d_{\on{inn}}(\mc{P}|G) = \frac{1}{|G|}\sum_{A\in\mc{P},A\subset G}|A| $$
$$ d_{\on{out}}(\mc{P}|G) = \frac{1}{|G|}\sum_{A\in\mc{P},A\subset G}|A| $$

We define the inner (resp. outer) densities of $\mc{P}$ with respect to the gauge $G$ as
$$ d_-(\mc{P},G,o) = \liminf_{\lambda\to\infty} d_{\on{inn}}(\mc{P}|\lambda G) $$
$$ d_+(\mc{P},G,o) = \limsup_{\lambda\to\infty} d_{\on{out}}(\mc{P}|\lambda G) $$
where $\lambda G$ is the image of $G$ under a homothety of scale $\lambda$ fixing $o$.

Call $(G,o)$ a gauge for the density. Then the packing density $\delta$ of $\Omega$ is defined to be the supremum of the outer densities $d_+(\mc{P},G,o)$ over all packings of $\Omega$ and all choices of gauge $(G,o)$.

According to a theorem of Groemer \cite{Groemer1986}, c.f. also section 2 of \cite{Toth1993}, for every compact domain $\Omega$ in $\mb{R}^n$, there exists a packing $\mc{P}$ by congruent copies of $\Omega$ such that 
$$ d_+(\mc{P},G,o) = d_-(\mc{P},G,o) = \delta $$
for every gauge $(G,o)$.

Therefore we may choose a suitable gauge pair: $([-1/2,1/2]^n,0)$. Then
$$ \delta = \lim_{\lambda\to\infty} d_{\on{inn}}(\mc{P},\lambda G) = \lim_{\sigma\to\infty} \frac{1}{\sigma^n V}\sum_{A\in\mc{P},A\subset\lambda\Omega}|A|. $$
In view of the fact that every $A$ is a copy of $\Omega$, we have
$$ \delta = \lim_{\sigma\to\infty} \sigma^{-n}\sum_{A\in\mc{P},A\subset\lambda\Omega} V = \lim_{\sigma\to\infty} \frac{V N(\sigma)}{\sigma^n}. $$

Therefore, 
$$ \lim_{\sigma\to\infty} \frac{\sigma^n}{N(\sigma)} = \frac{V}{\delta}. $$
\end{proof}

\bibliographystyle{plain}
\bibliography{packing_refs}

\end{document}